\documentclass{amsart}%
\usepackage{amsmath}
\usepackage{amsfonts}
\usepackage{amssymb}
\usepackage{graphicx}%

\usepackage[utf8]{inputenx}  
\usepackage[T1]{fontenc}  

\usepackage{mathscinet}

\DeclareMathOperator{\dom}{dom\!}
\DeclareMathOperator{\cod}{cod\!}
\DeclareMathOperator{\supp}{supp\!}

\newtheorem{theorem}{Theorem}

\newtheorem{lemma}[theorem]{Lemma}

\newtheorem{proposition}[theorem]{Proposition}

\title[No P-points]{An easy proof that there may be no P-points}                                                   

\author{David Chodounský}
\address{Institute of Mathematics of the Czech Academy of Sciences,
Žitná~25, Praha~1, Czech Republic.
}
\email{chodounsky@math.cas.cz}
\thanks{The first author was supported by project 26-22502S of the Czech Science Foundation (GAČR) and the Czech Academy of Sciences CAS (RVO 67985840).}

\author{Osvaldo Guzm\'{a}n}
\address{Centro de Ciencias Matem\'{a}ticas, UNAM.
}
\email{oguzman@matmor.unam.mx}
\thanks{The second author was supported by the PAPIIT
grant IA 104124 and the SECIHTI grant CBF2023-2024-903.}

\author{Jonathan Verner}
\address{HID Global Corporation, Thamova 183/11, 186 00 Prague, Czech Republic.
}
\email{jonathan.verner@matfyz.cz}
\thanks{The third author would like to thank the OZP for keeping him alive for a number of years.}


\begin{document}

\begin{abstract}
\noindent
We provide an easy proof that there are no P-points in the models obtained by adding many Silver reals with countable support to a model of~\textsf{CH}.
\end{abstract}

\maketitle

\section{Introduction}

\noindent
The existence of ultrafilters on natural numbers with special
combinatorial properties is a major topic in set theory and set-theoretic topology. 
Ultrafilters with special combinatorial properties are often hard to construct, and more often than not, their existence is independent from the axioms of~\textsf{ZFC.}
Kunen started this trend in~\cite{KunenSomePoints}, where he proved that there are no Ramsey ultrafilters in the random model. 
Not long after Shelah found a model with no P-points 
(see~\cite{ProperandImproper, Wimmers}) 
and Miller showed that there are no Q-points in the Laver model and the Miller model, see~\cite{NoQpointsMiller, Rat}. 
Since then many independence theorems
regarding the non-existence of special types of ultrafilters have been proved, see~\cite{Brendlenowheredense, UltraSurvey, AboveFsigma, IsbellProblem, TherearenoPpointsinSilverExtensions,PFiltersCohenRandomLaver, OrderedUnionvsSelective,    ShelahNowheredense}
among many others.


The class of P-points is the one that has received the most attention.
This is because of their usefulness, nice combinatorial properties, and good behavior with respect to forcing iteration. 
Shelah's construction of a model without P-points as presented in~\cite{ProperandImproper} and \cite{Barty} 
(which is not the same as in \cite{Wimmers}) proceeds as follows:
for each \textsf{P}-point $\mathcal{U}$, a forcing $\mathbb{P}^{\omega}(\mathcal{U})$ is defined. 
This forcing is proper, $\omega^{\omega}$-bounding and
\emph{kills} $\mathcal{U}$; the ultrafilter $\mathcal{U}$ cannot be extended to a \textsf{P}-point in the generic extension obtained by forcing with $\mathbb{P}^{\omega}(\mathcal{U})$.
Moreover, this remains true in any further extension by a proper, $\omega^{\omega}$-bounding forcing. 
A model without P-points is then
obtained by iterating forcings of type $\mathbb{P}^{\omega}(\mathcal{U})$, 
killing each \textsf{P}-point one at a time.


In~\cite{TherearenoPpointsinSilverExtensions} it was proved that \emph{the Silver forcing} (denoted by $\mathbb{SL}$) kills all ultrafilters in a similar way, 
although here the preservation property only holds for forcings with the Sacks property. 
While $\mathbb{P}^{\omega}(\mathcal{U})$ only kills $\mathcal{U}$,
the Silver forcing kills every ultrafilter in the ground model.
More precisely, $\mathbb{SL}$ can be decomposed as an iteration of first adding a generic selective ultrafilter $\mathcal{S}$, followed by forcing with $\mathbb{P}(\mathcal{S})$, which is a variant of Shelah's poset $\mathbb{P}^{\omega}(\mathcal{S})$. 
The forcing $\mathbb{P(\mathcal{S})}$ kills all P-points that are not near coherent with~$\mathcal{S}$, in particular all P-points of the ground model. 

It follows from a reflection argument that there are no P-points in the Silver model; the model obtained by adding $\omega_{2}$ Silver reals to a model of \textsf{CH}.
This is essentially the only known model without P-points that is obtained with a definable forcing. Furthermore, the argument from~\cite{TherearenoPpointsinSilverExtensions} also shows that there are no P-points in models obtained by forcing with countable support products of Silver forcing, not just iterations. 
In particular, the non-existence of P-points is consistent with the continuum being arbitrarily large. 
The purpose of this note is to present a new proof of this result, which is much easier than the one presented in~\cite{TherearenoPpointsinSilverExtensions}. 
Moreover, in our opinion, this is the simplest known proof of non-existence of P-points in any model, 
regardless of whether the continuum is $\omega_{2}$ or larger.


\section{Notation and Preliminaries}

\noindent
All ultrafilters are assumed to be non-principal. For $\mathcal{P} \subseteq {\left[  \omega\right]} ^{\omega}$ and $A\subseteq\omega$, we say that
$A$ is \emph{a pseudointersection} of $\mathcal{P}$ if it is almost contained in all elements of $\mathcal{P}$, 
i.e.\ $A \setminus P$ is finite for every $P \in \mathcal P$. 
An ultrafilter $\mathcal{U}$ on $\omega$ is
a \emph{P-point} if every countable subfamily of $\mathcal{U}$ has a pseudointersection in $\mathcal{U}$. 
Let $A$ and $B$ be two sets. 
The expression $f;A\to B$ denotes that $f$ is a partial function from $A$ to $B$. 
For $p;\omega\to2$ the domain of $p$ is denoted by
$\dom\left(  p\right)$, and $\cod\left(  p\right)$ is $\omega\setminus \dom \left(  p\right)$.
For convenience, we will write $p^{-1}\left(  1\right)$ instead of $p^{-1}\left[  \left\{  1\right\} \right]$. o
The \emph{Silver forcing} (denoted by $\mathbb{SL}$) 
consists of all partial functions $p;\omega\to2$ such that $\cod\left(  p\right)$ is infinite.
The order $q\leq p$ is simply $p\subseteq q$. 
If $G\subseteq\mathbb{SL}$ is a generic filter, the \emph{Silver generic real} is defined as $r_{\mathrm{gen}} = \bigcup G$, which is a total function from
$\omega$ to $2$. 

It is well-known that Silver forcing is proper and
$\omega^{\omega}$-bounding 
(this means that every function in $\omega^{\omega}$ in the generic extension is pointwise bounded by a ground model function). 
It is easy to see that the set of all $p\in\mathbb{SL}$ such that $p^{-1}\left( 1\right)$ is infinite is open dense. 
{From now on we assume that every element of} $\mathbb{SL}$ {has this property.} 
For more on Silver forcing, the reader may consult~\cite{Combinatorial2}. 
For a set of ordinals $X$ we denote by $\mathbb{SL}^{X}$ the set of all functions $p;X\to\mathbb{SL}$ with finite or countably infinite domain, 
which we denote by $\supp \left( p\right)$. 
Define $q\leq p$ if $\supp\left( p\right)  \subseteq \supp \left( q\right)$ and $q\left( \alpha\right) \leq p\left( \alpha\right)$ for all $\alpha\in \supp \left( p\right)$. 
It is not hard to prove that $\mathbb{SL}^{X}$ is an $\omega^{\omega}$-bounding proper forcing 
and \textsf{CH} implies that it has the $\omega_{2}$ chain condition, see~\cite{BaumgartnerAlmostDisjoint, Formalismclassofforcings}.

For two conditions $q$ and $r$ in a forcing $\mathbb{P},$ by $q\parallel r$ we
denote that $q$ and $r$ are compatible.

\section{Two Folklore results \label{Two}}

\noindent
Let $\kappa$ be an uncountable cardinal.
We say that $\mathcal{S}\subseteq {\left[\kappa\right]}^{\omega}$ 
is a \emph{stick family for} $\kappa$ if
every $A\in{\left[\kappa\right]}^{\kappa}$ contains an element of $\mathcal{S}$. 
For applications and results on stick families, 
see e.g.~\cite{BaumgartnerAlmostDisjoint, Stick, BarelyIndependent}. 
We are concerned when the ground model family ${\left[  \kappa\right]}^{\omega}$ remains a
stick family in a generic extension. The following result is easy to prove.

\begin{proposition}[\textsf{CH}]\label{Stick family}
Let $\kappa$ be a cardinal. $\mathbb{SL}^{\kappa}$ forces that
$V\cap {\left[\omega_{2}\right]}^{\omega}$ is a stick family for $\omega_{2}$.
\end{proposition}
\begin{proof}
Let $p\in\mathbb{SL}^{\kappa}$ and $\dot{B}$ be such that $p\Vdash \dot{B}\in\left[\omega_{2}\right]^{\omega_{2}}$. 
Find $\left\{\left(p_{\alpha},\beta_{\alpha}\right)
\mid\alpha\in\omega_{2}\right\}$ such that for every $\alpha, \gamma\in \omega_{2}$
\begin{enumerate}
\item $p_{\alpha}\leq p$,
\item $p_{\alpha}\Vdash \beta_{\alpha}\in\dot{B}$,
\item $\beta_{\alpha}\neq\beta_{\gamma}$ whenever $\alpha\neq\gamma$.
\end{enumerate}

Since we assume \textsf{CH}, 
we can apply the $\Delta$-system lemma
(Lemma~III.6.15 of~\cite{Kunen}) to $\{\supp\left(  p_{\alpha}\right)  \mid\alpha\in\omega_{2}\}\subseteq\left[  \kappa\right]^{{\leq}\omega}$ 
and find $W\in{\left[\omega_{2}\right]}^{\omega_{2}}$ and
$R\in{\left[\kappa\right]}^{{\leq}\omega}$ 
such that $\{\supp\left(p_{\alpha}\right)  \mid\alpha\in W\}$ is a $\Delta$-system with root
$R$. 
Since $\mathbb{SL}^{R}$ has size at most $\omega_{1}$ 
(exactly $\omega_{1}$ in case $R\neq\emptyset$), 
there is $Y\in{\left[W\right]}^{\omega_{2}}$ such that $p_{\alpha}\restriction R = p_{\gamma}\restriction R$ for all $\alpha,\gamma\in Y$. 
Let $X\subseteq Y$ be countable and define $q = \bigcup\limits_{\alpha\in X} p_{\alpha}$. 
It follows that $q\Vdash \left\{\beta_{\alpha}\mid\alpha\in
X\right\}  \subseteq\dot{B}$.
\end{proof}

In fact, a stronger
statement holds; 
Baumgartner proved in~\cite{BaumgartnerAlmostDisjoint} that $\mathbb{SL}^{\kappa}$ forces that
$V\cap{\left[\omega_{1}\right]}^{\omega}$ is a stick family for $\omega_{1}$. 
We will use the version for $\omega_{2}$ since its proof is
considerably simpler than the one for $\omega_{1}$ and sufficient for our purposes.

As in~\cite{Kunen}, $V^{\mathbb{P}}$ denotes the class of all $\mathbb{P}$-names for a partial order $\mathbb{P}$.
Let $\kappa$ be a cardinal,
$\mathcal{\dot{U}}$ an $\mathbb{SL}^{\kappa}$-name for an ultrafilter on $\omega$, and $X\subseteq\kappa$. 
The \emph{restriction of} $\mathcal{\dot{U}}$
\emph{to} $\mathbb{SL}^{X}$, 
denoted $\mathcal{\dot{U}}_{X}$, is the name for the family:
\[
\{\dot{A}\left[ G_{X}\right]  \mid\dot{A}\in V^{\mathbb{SL}^{X}},\,\exists p\in G_{X}(p\Vdash \dot{A}\in\mathcal{\dot{U}})\}
\]
where $\dot{G}_{X}$ is the canonical name for the
generic filter on $\mathbb{SL}^{X}$. 
We will say that $X$ is $\mathcal{\dot{U}}$\emph{-suitable} if $\mathbb{SL}^{X}$ forces that $\mathcal{\dot{U}}_{X}$
is an ultrafilter. The following may be considered folklore:

\begin{lemma}[\textsf{CH}]\label{Suitable}
Let $\kappa$ be a cardinal, $p\in\mathbb{SL}^{\kappa}$ and
$\mathcal{\dot{U}}$ an $\mathbb{SL}^{\kappa}$-name for an ultrafilter on $\omega$. 
There is $X\in\left[  \kappa\right]  ^{\leq\omega_{1}}$ that is $\mathcal{\dot{U}}$-suitable. 
\end{lemma}

\begin{proof}
Construct $\left\{M_{\alpha}\mid\alpha\leq\omega_{1}\right\}$ 
a continuous increasing chain of elementary submodels of $H(\lambda)$ 
(for some large enough $\lambda$) with the following properties:
\begin{enumerate}
\item $\mathcal{\dot{U}},\kappa\in M_{0}$.
\item If $\alpha<\omega_{1}$, then $M_{\alpha}$ is countable.
\item ${\left[ M_{\omega_{1}}\right]}^{\omega}\subseteq M_{\omega_{1}}$.
\end{enumerate}
This is easy to do using \textsf{CH}.
Let $X=M_{\omega_{1}}\cap\kappa$.
Note that $\mathbb{SL}^{X}\subseteq M_{\omega_{1}}$ 
and every countable subset of
$\mathbb{SL}^{X}$ is also in $M_{\omega_{1}}$. 

We will now prove that $X$ is $\mathcal{\dot{U}}$-suitable. 
To see this let $p\in\mathbb{SL}^{X}$ and $\dot{B}$ be an $\mathbb{SL}^{X}$-name for a subset of $\omega$. 
For every $n\in\omega$ find a maximal antichain of conditions $A_{n}$ in $M_{\omega_1}$
such that every element of $A_{n}$ decides if $n$ is an element of $\dot{B}$. 
Since $\mathbb{SL}^{X}$ is proper, 
there is $q \in \mathbb{SL}^{X}$, $q\leq p$ such that each 
$q\parallel A_{n}=\left\{r\in A_{n}\mid q\parallel r\right\}$
is countable.
By the previous remarks, we can find
$\alpha<\omega_{1}$ such that $q$ and $\left\{q\parallel A_{n}\mid n\in \omega\right\}$ belong to $M_{\alpha}$. 
In particular, $q$ forces that $\dot{B}$ is equivalent to an $\mathbb{SL}^{X}$-name in $M_{\alpha}$, 
so any $(M_{\alpha}, \mathbb{SL}^{X})$-generic filter containing $q$ decides if
$\dot{B}$ is in $\mathcal{\dot{U}}$.
\end{proof}

For our proof in the next section
it is only sufficient to know that for every ground model set, 
the membership in $\mathcal{\dot{U}}$ is decided in 
$V\left[G_{X}\right]$, a fact which is even easier to prove.

\section{No more P-points}
\label{No more Ppoints}

\noindent
This section contains our main result: Adding at least 
$\omega_{2}$ Silver reals with countable support over a model of \textsf{CH}
provides a model with no P-points.

For every partial function $r;$ $\omega\to 2$ with $r^{-1}\left( 1\right)$ infinite 
define the interval partition $P\left(  r\right)  =\left\{  P_{n}\left( r\right)  \mid n\in \omega\right\}$
where $P_{n}(r) = \left\{\ell \in \omega \mid \left| \ell \cap r^{-1}(1) \right| =n\right\}$. 
Interval partitions induced by Silver generic reals are particularly interesting and will be fundamental for our main theorem. 
Fix a regular cardinal $\kappa>\omega_{1}$.
For $\alpha<\kappa$ and $i<2$ denote by $\dot{r}_{\alpha}$ 
the name of the $\alpha$-generic real added by
$\mathbb{SL}^{\kappa}$, 
and $\dot{D}_{i}\left(\alpha\right)$ is the
$\mathbb{SL}^{\kappa}$-name of $%
{\bigcup} \left\{ P_{2m+i}\left(  \dot{r}_{\alpha}\right)  \mid m\in\omega\right\}$.

\begin{proposition}[\textsf{CH}]\label{No-pseudoint}
Let $\mathcal{\dot{U}}$ be an $\mathbb{SL}^{\kappa}$-name for a nonprincipal ultrafilter on $\omega$, 
$X\in{\left[ \kappa\right]}^{\leq\omega_{1}}$ an
$\mathcal{\dot{U}}$-suitable set, $B\in\left[  \kappa\setminus X\right]^{\omega}$, and $i\in2$.
If $G\subseteq\mathbb{SL}^{\kappa}$ is a generic filter, 
then every pseudointersection of $\{\dot{D}_{i}\left(  \alpha\right) \mid\alpha\in B\}$ is disjoint with some element of $\mathcal{\dot{U}}_{X}\left[  G\right]$. 
\end{proposition}

\begin{proof}
Let $B=\left\{\alpha_{n}\mid n\in\omega\right\}$ and $p\in\mathbb{SL}^{\kappa}$, 
we can assume that $B\subseteq\supp\left(p\right)$.
Let $\dot{Z}$ be a name of a pseudointersection of 
$\{\dot{D}_{i}\left(\alpha_{n}\right)\mid n\in\omega\}$. 
Since $\mathbb{SL}^{\kappa}$ is $\omega^{\omega}$-bounding, we may assume there is a ground
model increasing function $f \colon \omega\to\omega$ such that
$p\Vdash \dot{Z}\setminus f\left(n\right)\subseteq\dot{D}_{i}\left(\alpha_{n}\right)$ 
for all $n\in\omega$.
We now find an interval partition $\left\{
E_{n}\mid n\in\omega\cup\left\{  -1\right\}  \right\}  $ such that the
following holds:
\begin{itemize}
\item $f\left(  n\right)<\min E_{2n}  $ for $n\in\omega$.
\item $E_{2n+j}\cap \cod\left(  p\left(  \alpha_{n}\right)  \right) \neq\emptyset$ for every $n\in\omega$ and $j\in2$.
\end{itemize}

Define $U_{0}= {\bigcup} \left\{ E_{2n+1}\mid n\in\omega\right\}$, this is a ground model set. 
Recall that $X$ is $\mathcal{\dot{U}}$-suitable, so the coordinates outside 
$X$ are not needed to decide whether $U_{0}$ is in $\mathcal{\dot{U}}$. 
Since $B\cap X=\emptyset$, 
we can find $p_{1}\leq p$, $p\restriction B=p_{1}\restriction B$ such that $p_{1}$ decides if $U_{0}$ is in $\mathcal{\dot{U}}$. 
If $p_{1}\Vdash U_{0}\in\mathcal{\dot{U}}$, 
let $U=U_{0}$ and $A_{n}=E_{n}$ for $n\in\omega\cup\left\{  -1\right\}$. 
In case $p_{1}\Vdash U_{0}\notin\mathcal{\dot{U}}$, let $A_{-1}=E_{-1}\cup E_{0}$ and
$A_{n}=E_{n+1}$ for $n\in\omega$. 
Define $U=\bigcup\left\{A_{2n+1}\mid n\in\omega\right\}$. 
In either case, we have the following:
\begin{itemize}
\item $\left\{A_{n}\mid n\in\omega\cup\left\{  -1\right\}  \right\}$ is an interval partition.
\item $U=\bigcup\left\{A_{2n+1}\mid n\in\omega\right\}$ and $p_{1} \Vdash U\in\mathcal{\dot{U}}$.
\item $f\left(  n\right)  < \min A_{2n}$ for every $n\in\omega$, so, in particular 
\newline $p_{1}\Vdash \dot {Z}\cap A_{2n+1}\subseteq\dot{D}_{i}\left(\alpha_{n}\right)$. 
\item $A_{2n}\cap\cod\left( p_{1}\left(  \alpha_{n}\right) \right)  \neq\emptyset$ for every $n\in\omega$.
\end{itemize}

Let $p_{2}\leq p_{1}$ be a condition such that for every $n\in\omega$ the following holds:
\begin{itemize}
\item $A_l \cap \cod\left(p_{2}\left(\alpha_{n}\right)  \right) =\emptyset$ for every $l<2n$.
\item $\left|A_{2n} \cap\cod\left(  p_{2}\left(  \alpha_{n}\right)  \right) \right| = 1$.
\item $ A_{2n+1} \cap \cod \left(p_{2}\left(  \alpha_{n}\right)  \right)=\emptyset$.
\end{itemize}

Let $a_{n}$ be the unique element of 
$A_{2n} \cap \cod\left(p_{2}\left(\alpha_{n}\right)\right)$ 
and define the set 
\[ H_{n} =A_{2n+1}\cap
{\textstyle\bigcup} \left\{  P_{2m+i}\left(  p_{2}\left(  \alpha_{n}\right)  \right)  \mid m\in\omega\right\}.\] 
It is easy to see that for every $q\leq p_{2}$ the following hold:
\begin{itemize}
\item If $q\left(\alpha_{n}\right)\left(a_{n}\right)  =0$, 
then $q\Vdash A_{2n+1}\cap\dot{D}_{i}\left(  \alpha_{n}\right)  = H_{n}$.
\item If $q\left(  \alpha_{n}\right)  \left(  a_{n}\right)  =1$, then
$q\Vdash A_{2n+1}\setminus\dot{D}_{i}\left(  \alpha_{n}\right)
=H_{n}$, 
and so, in particular, $q\Vdash H_{n} \cap\dot{D}_{i}\left(  \alpha_{n}\right)  =\emptyset$.
\end{itemize}

Define $H= {\bigcup}\{H_n \mid n \in \omega\}$,
note that $H$ is a ground model set. 
Once again, find
$p_{3}\leq p_{2}$ that decides if $H$ is an element of $\mathcal{\dot{U}}$ and
$p_{3}\restriction B=p_{2}\restriction B$. 
We will now proceed by cases;
first, assume that $p_{3}\Vdash H\notin\mathcal{\dot{U}}$. 
Let $q$ be an extension of $p_{3}$ such that $q\left(  \alpha_{n}\right) \left(  a_{n}\right)  =0$ 
for every $n\in\omega$. 
We know that $q\Vdash A_{2n+1}\cap\dot{D}_{i}\left(  \alpha_{n}\right)  =H_{n}$
for every $n\in\omega$. 
Thus, $q$ forces the following:
\begin{multline*} 
\dot{Z}\cap U = 
\bigcup\limits_{n\in\omega}(\dot{Z}\cap A_{2n+1}) \subseteq \bigcup \limits_{n\in\omega}(\dot{D}_{i}\left(  \alpha
_{n}\right)  \cap A_{2n+1}) = \bigcup\limits_{n\in\omega}H_{n}
=H \notin \mathcal{\dot{U}}.
\end{multline*}
So, $q\Vdash \dot{Z}\cap\left(U\setminus H\right)  =\emptyset$.
This finishes the proof for this case, 
since $U\setminus H$ is forced to be in $\mathcal{\dot{U}}$. 

We are left with the case $p_{3}\Vdash H\in\mathcal{\dot{U}}$. 
Let $q$ be an extension of $p_{3}$ such that
$q\left( \alpha_{n}\right) \left(a_{n}\right)=1$ for every $n\in\omega$. 
Now $q\Vdash H_{n}\cap\dot{D}_{i}\left(  \alpha_{n}\right)  =\emptyset$ for every $n\in\omega$. 
It follows that $q$ forces the following:
\begin{multline*}
\dot{Z}\cap H = \bigcup\limits_{n\in\omega}(\dot{Z}\cap H_{n})
\subseteq \bigcup\limits_{n\in\omega}(\dot{D}_{i}\left(  \alpha_{n}\right)  \cap A_{2n+1}\cap H_{n}) = \bigcup\limits_{n\in\omega}(\dot{D}_{i}\left(\alpha_{n}\right) \cap H_{n})=\emptyset.
\end{multline*}
\end{proof}

The main result now follows easily.

\begin{theorem}[\textsf{CH}]\label{teorema}
Let $\kappa>\omega_{1}$ be a regular cardinal. 
$\mathbb{SL}^{\kappa}$ forces that there are no P-points.
\end{theorem}

\begin{proof}
Let $\mathcal{\dot{U}}$ be an $\mathbb{SL}^{\kappa}$-name for an ultrafilter and $p\in\mathbb{SL}^{\kappa}$. 
By Proposition~\ref{Suitable} we can find an $\mathcal{\dot{U}}$-suitable set $X\in{\left[\kappa\right]}^{\omega_{1}}$. 
For every $\alpha\in\kappa$ let $\dot{e}_{\alpha}$ be an $\mathbb{SL}^{\kappa}$-name for which $p\Vdash \dot{D}_{\dot{e}_{\alpha}}\left(\alpha\right)  \in\mathcal{\dot{U}}$. 
Pick $W\in{\left[\kappa\right]}^{\omega_{2}}$ such that $W\cap X=\emptyset$. 
By extending $p$ if necessary, we may assume that there is $i\in2$ such that $p$ forces that
$\left\{\alpha\in W\mid\dot{e}_{\alpha}=i\right\}$ 
has size $\omega_{2}$.
Apply Proposition~\ref{Stick family} to find a ground model countable set $B\subseteq W$ 
and $q\leq p$ such that $q\Vdash \dot{e}_{\alpha}=i$ 
for every $\alpha\in B$. 
Proposition~\ref{No-pseudoint} entails that $q$ forces that 
$\{\dot{D}_{i}\left(\alpha\right) \mid \alpha\in B\}$ 
is a countable subset of $\mathcal{\dot{U}}$ 
that does not have a pseudointersection in $\mathcal{\dot{U}}$.
\end{proof}

\section{Towards Tukey Top}
\label{TTT}

\noindent
Let $\mathcal{U}$ be an ultrafilter on  natural numbers.
We say that
$\mathcal{U}$ is \emph{Tukey top} if there is $\mathcal{W}\in{\left[\mathcal{U}\right]}^{\mathfrak{c}}$ 
such that 
${\bigcap}\mathcal{W}_{1}\notin\mathcal{U}$ 
for every $\mathcal{W}_{1}\in{\left[\mathcal{W}\right]}^{\omega}$ 
(here $\mathfrak{c}$ denotes the cardinality
of the set of real numbers). 
It was a long standing problem of Isbell whether
\textsf{ZFC} implies the existence of a non-Tukey top ultrafilter. 
The problem was recently solved in the negative by Cancino and~Zapletal in~\cite{IsbellProblem}. 
However, it is still not known whether every ultrafilter is Tukey top in models obtained by adding Silver reals, either by iteration or countable support product. 
The proof of Theorem~\ref{teorema}
originated from our attempt to show that there are no Tukey-top ultrafilters in such models. 
Although we were ultimately unable to prove this result, 
we believe that we got close. 
Let us sketch our attempted proof.

Assume~\textsf{CH}, let $\kappa$ be a regular uncountable cardinal, and $\mathcal{\dot{U}}$ be an $\mathbb{SL}^{\kappa}$-name for an ultrafilter.
Find $X,i$ and $W$ as in the proof of Theorem~\ref{teorema},
only this time make sure $W$ has size $\kappa$. 
Proposition~\ref{No-pseudoint} implies that if
$B\in\left[  W\right]  ^{\omega}\cap V,$ then $\bigcap \{D_{i}\left(\dot{r}_{\alpha}\right) \mid {\alpha\in
B}\}$ cannot be in $\mathcal{\dot{U}}$. 
The problem is that we do not know what to do in case $B\notin V$. 
We believed that this argument can be refined to take care of all countable subsets of $W$;
this would imply that $\mathcal{\dot{U}}$ is forced to be Tukey top.

\section*{acknowledgement}
We are grateful to Stevo Todorcevic, Michael Hrušák, Juris
Stepr\={a}ns, Jonathan Cancino, and Jindřich Zapletal for insightful discussions on the topics of this work. 
We also thank the referee for her or his remarks that improved the paper.

\bibliographystyle{plain}
\bibliography{Bibliografia}

\end{document}